\documentclass[12pt]{amsart}

\usepackage{amssymb,latexsym}

\usepackage{enumerate}
\usepackage{amsmath}

\makeatletter

\@namedef{subjclassname@2010}{

  \textup{2010} Mathematics Subject Classification}

\makeatother
\newtheorem{thm}{Theorem}[section]
\newtheorem{Con}[thm]{Conjecture}

\newtheorem{cor}[thm]{Corollary}
\newtheorem{lem}[thm]{Lemma}
\newtheorem{pro}[thm]{Proposition}
\theoremstyle{definition}

\numberwithin{equation}{section}

\newcommand{\ep}{\varepsilon_d}

\newcommand{\re}{\textup{Re}}
\newcommand{\im}{\textup{Im}}
\newcommand{\D}{\mathcal{D}}
\newcommand{\N}{\mathcal{N}}

\newcommand{\sums}{\sum_{\substack{d\in \D\\ \ep \leq x}}}
\newcommand{\sumf}{\sideset{}{^\flat}\sum}

\begin{document}

\baselineskip=17pt

\title[Large moments and extreme values of class numbers]
{Large moments and extreme values of class numbers of indefinite binary quadratic forms}

\author{Youness Lamzouri}

\address{Department of Mathematics and Statistics,
York University,
4700 Keele Street,
Toronto, ON,
M3J1P3
Canada}

\email{lamzouri@mathstat.yorku.ca}

\date{}

\begin{abstract}  Let $h(d)$ be the class number of indefinite binary quadratic forms of discriminant $d$, and let $\ep$ be the corresponding fundamental unit. In this paper, we obtain an asymptotic formula for the $k$-th moment of $h(d)$ over positive discriminants $d$ with $\ep\leq x$, uniformly for real numbers $k$ in the range $0<k\leq (\log x)^{1-o(1)}$. This improves upon the work of Raulf, who obtained such an asymptotic for a fixed positive integer $k$. We also investigate the distribution of large values of $h(d)$ when the $d$'s are ordered according to the size of their fundamental units $\ep$. In particular, we show that the tail of this distribution has the same shape as that of class numbers of imaginary quadratic fields ordered by the size of their discriminants. As an application of these results, we prove that there are many positive discriminants $d$ with class number $h(d)\geq (e^{\gamma}/3+o(1))\cdot \ep  (\log\log \ep)/\log \ep$, a bound that we believe is best possible. We also obtain an upper bound for $h(d)$ that is twice as large, assuming the generalized Riemann hypothesis.

\end{abstract}


\thanks{ The author is partially supported by a Discovery Grant from the Natural Sciences and Engineering Research Council of Canada.}

\maketitle

\section{Introduction}

Let $\D=\{d\in \mathbb{N}: d\equiv 0 \text{ or } 1\pmod 4, \text{ and } d \text{ is not a square}\}$ be the set of positive discriminants. For $d\in \D$ let $h(d)$ be the number of equivalence classes of primitive indefinite binary quadratic forms $ax^2+bxy+cy^2$ of discriminant $d=b^2-4ac$. The study of the class numbers $h(d)$ is an important problem in number theory, and its rich history goes back to the work of Gauss. In particular, in his Disquistiones Arithmeticae, Gauss 
conjectured that there are infinitely many $d\in \D$ for which $h(d)=1$, a problem that is still open. The difference between this case and the easier one of positive definite forms (which corresponds to negative discriminants) is the role played by the fundamental unit $\ep$  when $d>0$. Recall that $\varepsilon_d= (t_d+u_d\sqrt{d})/2,$
where $t_d$ and $u_d$ are the smallest positive integer solutions to the Pell equation $t^2-du^2=4$. Indeed for $d\in \D$, Dirichlet's class number formula asserts that 
\begin{equation}\label{ClassNumberFormula}
h(d)=\sqrt{d}\cdot \frac{L(1,\chi_d)}{\log \ep},
\end{equation}
where $\chi_d=\left(\frac{d}{\cdot}\right)$ is the Kronecker symbol, and $L(s,\chi_d)$ is the Dirichlet $L$-function associated to $\chi_d$. Gauss \cite{Ga} asserted without proof that
$$\sum_{\substack{ d\in \D\\ d\leq x}} h(d)\log \ep\sim\frac{\pi^2}{18 \zeta(3)} x^{3/2},$$
an asymptotic formula that was later proved by Siegel \cite{Si}. However, Gauss also observed that unlike the quantity $h(d)\log \ep$, the class number $h(d)$  appears to behave rather erratically on average. In fact, for almost two centuries following Gauss's Disquistiones Arithmeticae no asymptotic formula was proposed, let alone proved, for the average of $h(d)$ over $d$, until Hooley \cite{Ho} formulated a conjecture for this average.  Using a refined analysis of the behaviour of the fundamental units $\ep$, Hooley \cite{Ho} conjectured that
$$\sum_{\substack{ d\in \D\\ d\leq x}} h(d)\sim \frac{25}{12\pi^2}x(\log x)^2.$$
Proving such an asymptotic appears to be one of the most difficult problems in this theory. The main difficulty comes from the erratic behaviour of the fundamental unit $\ep$ which can be as small as $\sqrt{d}$ but might as well be as large as $\exp(\sqrt{d})$. 

A few years prior to Hooley's work, Sarnak \cite{Sa1} observed that this problem can be solved if one orders the discriminants $d$ instead by the size of their fundamental units $\ep$. This ordering makes the problem easier, since we are favouring those $d$ with small $\ep$. Using the Selberg trace formula, Sarnak proved that
$$\sums h(d) =\text{Li}(x^2) + O\left(x^{3/2+\epsilon}\right),$$
and he also showed that
\begin{equation}\label{SarnakAsymp}
 \sums 1 \sim \frac{35}{16} x,
 \end{equation}
 and hence it follows that the average order of $h(d)$ over the discriminants $d\in \D$ with $\ep\leq x$ equals $8x/(35\log x)$.

Based on a subsequent work of Sarnak \cite{Sa2}, Raulf \cite{Ra} obtained a similar asymptotic for the average of $h(d)$ over fundamental discriminants $d$. Her approach avoids the use of the Selberg trace formula at the cost of a weaker error term. Later in \cite{Ra2}, Raulf generalized her method to obtain an asymptotic formula for the $k$-th moment of $h(d)$, for any fixed natural number $k$.  We improve on this result, by obtaining an asymptotic formula for the $k$-th moment of $h(d)$ over positive discriminants $d$ with $\ep\leq x$, uniformly for all real numbers $k$ in the range $0<k\leq (\log x)^{1-o(1)}$. Our approach is different, and relies on the methods of Granville and Soundararajan \cite{GrSo} and the author (\cite{La0} and \cite{La1}) for computing large moments of $L(1,\chi_d)$. Here and throughout we let $\log_j$ be the $j$-fold iterated logarithm; that is, $\log_2=\log\log$, $\log_3=\log\log\log$ and so on.
\begin{thm}\label{Main} Let $x$ be large. There exists a positive constant $B$ such that uniformly for all real numbers $k$ with $0<k \leq \log x/(B\log_2 x\log_3 x)$ we have 
$$\sums h(d)^k= \mathcal{H}(k)\cdot \int_{2}^{x} \left(\frac{t}{\log t}\right)^k dt + O\left(\frac{x^{k+1}}{(\log x)^k}\exp\left(-\frac{\log x}{200\log_2 x}\right)\right),$$
where the constant $\mathcal{H}(k)$ is defined in \eqref{TheH} below. Moreover,  we have
$$\mathcal{H}(k)= \prod_{p}\mathcal{H}_p(k),$$
where $\mathcal{H}_2(k)=\frac12+ O\left(\frac{2^k}{3^k}\right),$
and for $p\geq 3$ we have
$$ \mathcal{H}_p(k)=\left(\left(\frac12-\frac{3}{2p}\right)\left(1-\frac{1}{p}\right)^{-k}+ \left(\frac12-\frac{1}{2p}\right)\left(1+\frac{1}{p}\right)^{-k}+\frac{2}{p}\right)\left(1+ O\left(\frac{1}{(p-1)^k}\right)\right).
$$
\end{thm}
Note that 
$$ 
\int_{2}^{x} \left(\frac{t}{\log t}\right)^k dt = \frac{x^{k+1}}{(k+1)(\log x)^k}\left(1+ O\left(\frac{1}{\log x}\right)\right),
$$
uniformly in $k>0$, and hence we deduce the following corollary.
\begin{cor}\label{Main2} Let $x$ be large. There exists a positive constant $B$ such that uniformly for all real numbers $k$ with $0<k \leq \log x/(B\log_2 x\log_3 x)$ we have
$$\sums h(d)^k =\frac{\mathcal{H}(k)}{k+1}\cdot\frac{x^{k+1}}{(\log x)^k}\left(1+ O\left(\frac{1}{\log x}\right)\right).$$
\end{cor}
One can obtain similar results for the corresponding moments of class numbers over fundamental discriminants, by using the techniques of this paper together with \cite{Ra}.  Note that when $d$ is fundamental, $h(d)$ is the narrow class number of the real quadratic field $\mathbb{Q}(\sqrt{d})$, which equals the class number of $\mathbb{Q}(\sqrt{d})$ if the negative Pell equation $t^2-du^2=-4$ is solvable, and equals twice this class number if the negative Pell equation is not solvable. Hence, in order to translate our results to the setting of real quadratic fields, one needs to understand the difficult problem of the distribution of discriminants $d$ for which the negative Pell equation is solvable (see \cite{FoKl} for a reference on this problem). 

As an application of Theorem \ref{Main} we investigate large values of  $h(d)$ and their distribution, when the discriminants $d$ are ordered according to the size of their fundamental units. We remark here that this ordering causes the class numbers $h(d)$ to be distributed very differently than if one orders the $d$'s naturally according to their size. Indeed, note that the class number $h(d)$ is expected to be typically of size $\log\ep$ if we order by the size of $d$, while Theorem \ref{Main} shows that $h(d)$ is typically of size $\ep/\log\ep$ when we order the $d$'s by the size of $\ep$. This is due to the fact that this latter order favours those discriminants $d$ with small fundamental units $\ep$. 


It follows from the work of Littlewood \cite{Li} that assuming the generalized Riemann hypothesis GRH one can approximate $L(1, \chi_d)$ by a short Euler product over the primes $p\leq (\log d)^{2+o(1)}$. More precisely, assuming GRH one has (see Lemma 2.1 of \cite{GrSo})
\begin{equation}\label{TruncGRH}
L(1, \chi_d)= \prod_{p\leq (\log d)^2}\left(1-\frac{\chi_d(p)}{p}\right)^{-1}\left(1+O\left(\frac{\log_3 d}{\log_2 d}\right)\right).
\end{equation}
Moreover, Littlewood \cite{Li} conjectured that the shorter Euler product over the primes $p\leq \log d$ still serves as a good approximation for $L(1,\chi_d)$. More precisely we have
\begin{Con}[Littlewood]\label{LittlewoodConjecture}
$$L(1,\chi_d)\sim \prod_{p\leq \log d}\left(1-\frac{\chi_d(p)}{p}\right)^{-1}.$$
\end{Con}
In \cite{GrSo}, Granville and Soundararajan investigated the distribution of $L(1,\chi_d)$ and their results give strong support to this conjecture. 
Since $\sqrt{d}\leq \ep$, we deduce from the class number formula \eqref{ClassNumberFormula} that
\begin{equation}\label{CrudeBound}
h(d)\leq \begin{cases} \displaystyle{\left(2e^{\gamma}+o(1)\right) \frac{\ep\log\log \ep}{\log \ep}} &\text{ assuming GRH},\\
\displaystyle{\left(e^{\gamma}+o(1)\right) \frac{\ep\log\log \ep}{\log \ep}} &\text{ assuming Conjecture \ref{LittlewoodConjecture}},\\
\end{cases}
\end{equation}
where $\gamma$ is the Euler-Mascheroni constant. 
We shall prove that one can save a factor of $3$ over these bounds. This saving comes from the fact that the discriminants $d$ for which $\ep$ is very close to $\sqrt{d}$ satisfy $\chi_d(2)\neq 1$ and $\chi_d(3)\neq 1$.
\begin{thm}\label{GRHBound}
Let  $d$ be a large positive discriminant. Then we have 
$$
h(d)\leq \begin{cases} \displaystyle{\left(\frac{2e^{\gamma}}{3}+o(1)\right) \frac{\ep\log\log \ep}{\log \ep}} &\text{ assuming GRH},\\
 \displaystyle{\left(\frac{e^{\gamma}}{3}+o(1)\right) \frac{\ep\log\log \ep}{\log \ep}} &\text{ assuming Conjecture \ref{LittlewoodConjecture}}.\\
\end{cases}$$
\end{thm}
It follows from Theorem \ref{Main} that $h(d) (\log \ep)/\ep$ has a limiting distribution as $\ep\to \infty$. This also follows from the results of \cite{Ra2}. This distribution function is difficult to compute, due to complicated nature of its moments (see \eqref{TheH} below).  Nevertheless, using the saddle-point method, we are able to obtain a precise estimate for its large deviations. More specifically, using Theorem \ref{Main} we show, in a large uniform range, that the tail of this distribution has the same shape as that of class numbers of imaginary quadratic fields ordered by the size of their discriminants, which was derived by Granville and Soundararajan \cite{GrSo}. As a consequence, we prove that the second bound in Theorem \ref{GRHBound} is best possible, and is attained for many discriminants $d\in \D$ with $\ep \leq x$. 

\begin{thm}\label{Distribution}
Let $x$ be large. There exists a positive constant $C$, such that uniformly in the range $1\ll \tau \leq \log_2 x-\log_3 x-\log_4 x-C$, the proportion of positive discriminants $d$ with $\ep\leq x$ and such that
$$
h(d)\geq \frac{e^{\gamma}}{3} \frac{x}{\log x}\cdot \tau,
$$
equals 
\begin{equation}\label{TailDistributionClass}  \exp\left(-\frac{e^{\tau-A_0}}{\tau}\left(1+ O\left(\frac{1}{\sqrt{\tau}}\right)\right)\right),
\end{equation}
where 
\begin{equation}\label{SpecialConstant}
A_0:= \int_0^1\frac{\log\cosh(t)}{t^2}dt + \int_1^{\infty}\frac{\log\cosh(t)-t}{t^2}dt+1=0.8187\cdots.
\end{equation}

\end{thm}
\begin{cor}\label{OmegaResult}
There are at least $x^{1-1/\log\log x}$ discriminants  $d\in \D$ with $\ep\leq x$ such that
$$ h(d)\geq \left(\frac{e^{\gamma}}{3}+o(1)\right) \frac{\ep\log\log \ep}{\log \ep}.$$
\end{cor}

The paper is organized as follows. In Section 2, we prove the conditional result Theorem \ref{GRHBound}. In Section 3, we establish an asymptotic formula for the sum $\sum_{\ep\leq x} \chi_d(m) d^{k/2}$. This will be one of the main ingredients to compute the moments of $h(d)$ and prove Theorem \ref{Main}, which shall be completed in Section 4. In Section 5, we investigate the distribution of large values of $h(d)$ and prove Theorem \ref{Distribution}.

\section{A conditional bound for the class number in terms of the fundamental unit: proof of Theorem \ref{GRHBound}}

We shall only establish the first bound under GRH, since the second can be obtained along the same lines by assuming Conjecture \ref{LittlewoodConjecture} and replacing $(\log d)^2$ by $\log d$ in the argument below.

Assume GRH and let $d$ be a large positive discriminant. Then, it follows from \eqref{TruncGRH} that
\begin{align*}
 L(1,\chi_d)&\leq (1+o(1))\left(1-\frac{\chi_d(2)}{2}\right)^{-1}\left(1-\frac{\chi_d(3)}{3}\right)^{-1}\prod_{5\leq p\leq (\log d)^{2}}\left(1-\frac{1}{p}\right)^{-1}\\
&\leq \left(\frac{2e^{\gamma}}{3}+o(1)\right)\left(1-\frac{\chi_d(2)}{2}\right)^{-1}\left(1-\frac{\chi_d(3)}{3}\right)^{-1}\log\log d.\\
\end{align*}
Let $\ep=(t_d+u_d\sqrt{d})/2$. Since $t_d^2= u_d^2 d+4$, then one has $\sqrt{d}\leq \ep/u_d$. Thus, from the class number formula \eqref{ClassNumberFormula} we derive, assuming GRH
$$ h(d) \leq \left(\frac{2e^{\gamma}}{3}+o(1)\right)\frac{ \ep \log\log \ep}{\log \ep}\cdot E(d), $$
where 
$$E(d)= \frac{1}{u_d} \left(1-\frac{\chi_d(2)}{2}\right)^{-1}\left(1-\frac{\chi_d(3)}{3}\right)^{-1}.$$
To complete the proof, we will show that for all $d\in \D$ we have $E(d)\leq 1$. First, if $u_d=1$ then $d=t_d^2-4\not\equiv 1 \pmod 8$ and $d\not\equiv 1\pmod 3$. Therefore we have $\chi_d(2), \chi_d(3)\in \{-1, 0\}$ and thus $E(d)\leq 1$ in this case. Next, if $u_d=2$ then $t_d=2m$ is even and we deduce that $d=m^2-1$. Then, similarly to the previous case, one has $d\not\equiv 1 \pmod 8$ and $d\not\equiv 1\pmod 3$, and hence $E(d)\leq 1$. Finally if $u_d\geq 3$ then
$$ E(d)\leq \frac{1}{3}\left(1-\frac{1}{2}\right)^{-1}\left(1-\frac{1}{3}\right)^{-1}=1,$$
as desired.

\section{An asymptotic formula for the sum $\sum_{\ep\leq x} \chi_d(m)d^{k/2} $}
The purpose of this section is to prove the following result, which is one of the main ingredients to compute the moments of class numbers of indefinite binary quadratic forms.
\begin{thm} \label{AsympChar}
Let $m$ be a positive integer. Then, for any real number $0<k\leq \log x/\log_2 x$ we have
$$\sum_{\substack{d\in \D\\ \ep \leq x}} \chi_d(m) d^{k/2}= \frac{\mathcal{C}(k)}{k+1} \cdot g_k(m)\cdot x^{k+1} +O\left(m \cdot x^{k+49/50}\right),
$$
where
\begin{equation}\label{TheC}
\mathcal{C}(k):= \left(1+\frac{1}{2^{k+2}}+\frac{2}{4^{k+2}}+ \frac{4}{4^{k+2}(2^{k+2}-1)}\right)\prod_{p>2}\left(1+ \frac{2}{p^{k+2}-1}\right),
\end{equation}
and $g_k(m)$ is the multiplicative function defined by
$$g_k(2^a):= \frac{(-1)^a}{2}\left(1+\frac{2\left(1+(-1)^{a}\right)}{4^{k+2}(2^{k+2}-1)}\right)\left(1+\frac{1}{2^{k+2}}+\frac{2}{4^{k+2}}+ \frac{4}{4^{k+2}(2^{k+2}-1)}\right)^{-1}\text{ for } a\geq 1,$$
and for $p>2$ and $a\geq 1$ we have
$$g_k(p^a):=\begin{cases} -\frac{1}{p}\left(1+ \frac{2}{p^{k+2}-1}\right)^{-1}& \text{ if } a \text{ is odd},\\ \left(1-\frac{2}{p}\right)\left(1+ \frac{2(p-1)}{(p-2)(p^{k+2}-1)}\right)\left(1+ \frac{2}{p^{k+2}-1}\right)^{-1}& \text{ if } a \text{ is even}.\end{cases}$$
\end{thm}
Our starting point is a parametrization of the set of discriminants $d\in \D$ such that $\ep\leq x$. If $d$ is such a discriminant and $\ep=(t+u\sqrt{d})/2$ is its fundamental unit, then $d=(t^2-4)/u^2$. We shall prove that summing over $d\in \D$ such that $\ep\leq x$ amounts to essentially summing over the corresponding pairs $(t, u)$ in a certain range. This follows from the arguments in \cite{Sa1} and was used by Raulf \cite{Ra} to compute the first moment of class numbers over fundamental discriminants. 
Here and throughout we define
$$d(t, u):=\frac{t^2-4}{u^2}, \text{ for } t, u\in \mathbb{N}.$$
Then we have
\begin{lem}\label{Correspondence}
Let $x>0$ be large. There is a one-to-one correspondence between the pairs $(d, n)$ such that $ d\in \D$ and  $\ep^n\leq x$, and the pairs of positive integers $(t, u)$ such that $2<t\leq X$ and $d(t, u)\in \D$, where
$X$ is the unique solution to the equation 
\begin{equation}\label{EquationX}
\frac{X+\sqrt{X^2-4}}{2}=x.
\end{equation}
Moreover, the relation between a pair $(d, n)$ and its corresponding pair  $(t,u)$ is given by
$$ \ep^n= \frac{t+u\sqrt{d}}{2}.$$

\end{lem} 

\begin{proof} First, let $(t, u)$ be such that $2<t\leq X$ and $d(t, u)\in \D$. Hence $(t, u)$ is a solution to Pell's equation $t^2-u^2d=4$, with $d=d(t, u)$. Therefore, it follows from the theory of Pell's equation that there exists a positive integer $n$ such that 
$$ \ep^n= \frac{t+u\sqrt{d}}{2}=\frac{t+\sqrt{t^2-4}}{2}\leq x,$$ 
since $2< t\leq X$. On the other hand, if $(d, n)$ is such that $d\in \D$ and $\ep^n\leq x$, then  $\ep^n= (t+u\sqrt{d})/2$ where $(t, u)$ is a solution to Pell's equation $t^2-u^2 d=4$. Thus, for this pair $(t, u)$ we have $d(t, u)=d\in \D$ and
$2<t\leq X$ since $\ep^n= (t+\sqrt{t^2-4})/2\leq x$, as desired. 
 \end{proof}

Let $d\in \D$ and $\ep= (t+u\sqrt{d})/2$ be its fundamental unit, so that $d=d(t, u)$. If $t\equiv a \pmod{4u^2}$ then $t=a+4u^2\ell$ and hence $d=16 u^2\ell^2+ 8a\ell + d(a, u)$ for some $\ell\in \mathbb{N}$. Therefore, one has $\chi_d(m)= \left(\frac{P_{a, u}(\ell)}{m}\right)$, where $P_{a, u}$ is the quadratic polynomial defined by 
\begin{equation}\label{QuadPoly}
P_{a, u}(\ell):= 16 u^2\ell^2+ 8a\ell + d(a, u).
\end{equation}  
One of the key ingredients in the proof of Theorem \ref{AsympChar} is the evaluation of the following complete character sum
$$C_{m}(P_{a, u}):= \sum_{\ell=0}^{m-1}\left(\frac{P_{a, u}(\ell)}{m}\right).$$
\begin{lem}\label{CharacterSum}
Let $a, u$ be positive integers such that $d(a,u)\in \D$. Let $m$ be a positive integer and write $m=2^{e_1} p_2^{e_2}\cdots p_r^{e_r}$, where $e_1\geq 0$ and $e_j\geq 1$ for $2\leq j\leq r$. Also, let $m_0$ be the squarefree part of $m/2^{e_1}$. 

\noindent Then, if $(m_0, u)>1$ we have $C_{m}(P_{a, u})=0$. Moreover, if $(m_0, u)=1$ then we have 
$$\frac{C_{m}(P_{a, u})}{m}= b_{a, u}(m)\frac{(-1)^{\omega(m_0)}}{m_0} \prod_{\substack{2\leq j\leq r\\ e_j \text{ is even }}} \left(1-\frac{2}{p_j}\right) \prod_{\substack{ p_j \mid u\\ e_j \text{ is even }}} \left(1+\frac{1}{p_j-2}\right), $$
where $\omega(m_0)$ is the number of prime factors of $m_0$, and 
$b_{a, u}(m)= 1$ if $m$ is odd, and if $m$ is even we have 
$$ 
b_{a, u}(m)=\begin{cases} 0 & \text{ if } d(a, u)\equiv 0 \pmod 4\\ 1 & \text{ if } d(a, u)\equiv 1 \pmod 8\\
(-1)^{e_1} & \text{ if } d(a, u)\equiv 5 \pmod 8.
\end{cases}$$

\end{lem}

\begin{proof} Since  $C_{m}(P_{a, u})$ is a complete character sum then
$$
C_{m}(P_{a, u})= \sum_{\ell_1\bmod 2^{e_1}} \left(\frac{P_{a, u}(\ell_1)}{2^{e_1}}\right)
\prod_{j=2}^r\left(\sum_{\ell_j\bmod p_j^{e_j}} \left(\frac{P_{a, u}(\ell_j)}{p_j^{e_j}}\right)\right).
$$
First, since $P_{a, u}(\ell)\equiv d(a, u)\pmod 8$ for all $\ell$, then we obtain
$$
\sum_{\ell_1\bmod 2^{e_1}} \left(\frac{P_{a, u}(\ell_1)}{2^{e_1}}\right)
= 2^{e_1} \left(\frac{d(a, u)}{2}\right)^{e_1}= 2^{e_1} b_{a, u}(m),
$$
if $e_1\geq 1$. This equality is also valid if $e_1=0$.
On the other hand, for $2\leq j\leq r$ we have
$$
\sum_{\ell_j\bmod p_j^{e_j}} \left(\frac{P_{a, u}(\ell_j)}{p_j^{e_j}}\right)=\sum_{\ell_j\bmod p_j^{e_j}} \left(\frac{P_{a, u}(\ell_j)}{p_j}\right)^{e_j}.
$$
Since every $\ell\in \{1, 2,  \dots, p^{e}\}$ can be written in the form $i=b+ n p $ with $b\in \{1, 2, \dots, p\}$ and $n \in \{0, 1, \dots, p^{e-1}-1\}$ we deduce that for all $2\leq j\leq r$, we have
$$ \sum_{\ell_j\bmod p_j^{e_j}} \left(\frac{P_{a, u}(\ell_j)}{p_j}\right)^{e_j}= p_j^{e_j-1} \sum_{b_j\bmod p_j} \left(\frac{P_{a, u}(b_j)}{p_j}\right)^{e_j}.
$$
Now, it follows from Lemma 2.14 of \cite{Ra} that if $e_j$ is even we have
$$
\sum_{b_j\bmod p_j} \left(\frac{P_{a, u}(b_j)}{p_j}\right)^{e_j}
=\begin{cases} p_j-2 & \text{ if } p_j\nmid u\\
p_j-1 & \text{ if } p_j \mid u,
\end{cases}
$$
while if $e_j$ is odd we have
$$
\sum_{b_j\bmod p_j} \left(\frac{P_{a, u}(b_j)}{p_j}\right)^{e_j}
=\begin{cases} -1 & \text{ if } p_j\nmid u\\
0 & \text{ if } p_j \mid u,
\end{cases}
$$
Combining these estimates completes the proof.
\end{proof}
Note that when $m$ is even, the value of the character sum $C_{m}(P_{a, u})$ depends on whether $d(a, u)\equiv 0 \pmod 4$ or $d(a, u)\equiv 1, 5 \pmod 8$. Before completing the proof of Theorem \ref{AsympChar} we need the following Lemma which is a special case of Lemma 3.1 of \cite{Ra}.  

\begin{lem}[Lemma 3.1 of \cite{Ra}]\label{Raulf}
Let $\D_0=\{d \in \D: d\equiv 0\pmod 4\},\D_1= \{d \in \D: d\equiv 1\pmod 8\}$ and $\D_2= \{d \in \D: d\equiv 5\pmod 8\}$. Moreover, for  $i=0, 1, 2$, and a positive integer $u$ we define 
$$N_i(u):=|\{ 2<a\leq 4u^2+2: d(a, u) \in \D_i\}|.$$
Then, writing $u=2^{r_1}u_0$, where $r_1\geq 0$ and $u_0$ is odd, we obtain
\begin{align*}
 N_0(u)&=\begin{cases} 2\cdot 2^{\eta(u)} & \text{ if } r_1=0,\\ 
 4\cdot 2^{\eta(u)} & \text{ if } r_1=1,\\
 8\cdot 2^{\eta(u)} & \text{ if } r_1\geq 2,\end{cases} \ \ \ \ \ 
  N_1(u) =\begin{cases} 4\cdot 2^{\eta(u)} & \text{ if } r_1\geq 3,\\ 
0 & \text{ otherwise, }\end{cases}\\
 \text{ and }
 N_2(u)&=\begin{cases} 2\cdot 2^{\eta(u)} & \text{ if } r_1=0,\\ 
 4\cdot 2^{\eta(u)} & \text{ if } r_1\geq 3,\\
0 & \text{ otherwise, }\end{cases}\\
\end{align*}
where $\eta(u)=|\{p\geq 3: p|u\}|.$ 

\end{lem}

We now have the necessary ingredients to complete the proof of Theorem \ref{AsympChar}.

\begin{proof}[Proof of Theorem \ref{AsympChar}]
First, by Lemma \ref{Correspondence} we have
$$
\sum_{d\in \D}\sum_{\substack{n\geq 1\\ \ep^n \leq x}} \chi_d(m) d^{k/2}= \sum_{u\geq 1}\sum_{\substack{2<t\leq X\\  d(t, u)\in \D}} \left(\frac{d(t, u)}{m}\right) d(t, u)^{k/2}.
$$
Moreover, it follows from Lemma 2.1 of \cite{Ra} that
\begin{equation}\label{RaulfBound}
\left|\{(t, u): u>x^{\delta}, 2<t\leq x, \text{ and } d(t, u)\in D\}\right|\ll x^{1-\delta/5},
\end{equation}
for any $0<\delta<1/2$. Choosing $\delta=1/10$ and noting that $d(t, u)\leq t^2/u^2\leq x^{2(1-1/10)}$ and $X=x+O(1/x)$ we obtain 
$$\sum_{d\in \D}\sum_{\substack{n\geq 1\\ \ep^n \leq x}} \chi_d(m) d^{k/2}= \sum_{u\leq x^{1/10}}\sum_{\substack{2<t\leq x\\  d(t, u)\in \D}} \left(\frac{d(t, u)}{m}\right) d(t, u)^{k/2}+ O\left(x^{9k/10+49/50}\right).
$$
On the other hand, since $d\leq \ep^2$ one has
$$ \sum_{d\in \D}\sum_{\substack{n\geq 2\\ \ep^n \leq x}} \chi_d(m) d^{k/2} \ll x^{k/2}\sum_{2\leq n\ll \log x}\sum_{\substack{d\in D\\ \ep \leq x^{1/n}}} 1 \ll x^{(k+1)/2} \log x,$$
by \eqref{SarnakAsymp}. Therefore, we deduce that
\begin{equation}\label{DiophantineChange1}
\sum_{\substack{d\in \D\\ \ep \leq x}} \chi_d(m) d^{k/2}
=\sum_{u\leq x^{1/10}}\sum_{\substack{2<t\leq x\\ d(t, u)\in \D}} \left(\frac{d(t, u)}{m}\right) d(t, u)^{k/2}+O\left(x^{9k/10+49/50}\right).
\end{equation}
For $u\leq x^{1/10}$ fixed, consider the sum
\begin{equation}\label{DiophantineChange2}
\begin{aligned}
S_{u}(x)
&:=\sum_{\substack{2<t\leq x\\ d(t, u)\in \D}} \left(\frac{d(t, u)}{m}\right)d(t, u)^{k/2}\\
&= \sum_{a=3}^{4u^2+2} \sum_{\substack{2<t\leq x\\ t\equiv a \bmod 4 u^2 \\ d(t, u)\in \D}}\left(\frac{d(t, u)}{m}\right)d(t, u)^{k/2}.
\end{aligned}
\end{equation}
Now, if $t\equiv a\pmod{4u^2}$ and $2<t\leq x$, then $t= a+ 4u^2\ell $ where $0\leq \ell \leq (x-a)/(4 u^2)$. Furthermore, one has $ d(t, u)= 16 u^2\ell^2+ 8a\ell+ d(a, u)$, and hence $d(t, u)\equiv d(a, u)\pmod 8$ and $d(t, u) \in \D$  if and only if $d(a, u)\in \D$. Therefore, we deduce that
$$
S_{u}(x)= \sum_{\substack{2< a\leq 4u^2+2\\ d(a, u)\in \D}} \sum_{0\leq \ell \leq \frac{x-a}{4u^2}}\left(\frac{P_{a, u}(\ell)}{m}\right)P_{a, u}(\ell)^{k/2},
$$
where $P_{a,u}(\ell)$ is defined in \eqref{QuadPoly} above. For any integers $\ell\geq 1$ and $2<a\leq 4u^2+2$ we have $P_{a, u}(\ell)= 16 u^2\ell^2+ O(u^2\ell)$, and hence 
$P_{a, u}(\ell)^{k/2}= (4u \ell)^k + O\big(k(4u)^k \ell^{k-1}\big)$ if $k=o(\ell).$
Therefore, we obtain 
\begin{equation}\label{InnerSum1}
\begin{aligned}
S_{u}(x)
&=  \sum_{\substack{2<a\leq 4u^2+2\\ d(a, u)\in \D}} \sum_{\sqrt{x}\leq \ell \leq \frac{x-a}{4u^2}}\left(\frac{P_{a, u}(\ell)}{m}\right)P_{a, u}(\ell)^{k/2}+ O\left(4^ku^{k+2} x^{(k+1)/2}\right)\\
&= (4u)^k\sum_{\substack{2<a\leq 4u^2+2\\ d(a, u)\in \D}} \sum_{\sqrt{x}\leq \ell \leq \frac{x-a}{4u^2}}\left(\frac{P_{a, u}(\ell)}{m}\right)\ell^k +O\left(u^{2-k} x^{k+1/2}\right),\\
\end{aligned}
\end{equation}
since $k>0$ and $4^k=x^{o(1)}.$
The inner sum can be computed from the sum $\sum_{\ell \leq y}\left(\frac{P_{a, u}(\ell)}{m}\right)$ using partial summation.
Now observe that for a real number $y\geq 2$ we have
$$
\sum_{\ell \leq y}\left(\frac{P_{a, u}(\ell)}{m}\right)
= \frac{y}{m}\sum_{\ell=0}^{m-1}\left(\frac{P_{a, u}(\ell)}{m}\right) + O\left(m \right)= \frac{C_{m}(P_{a, u})}{m}\cdot y+O(m). 
$$
Therefore, by partial summation we deduce that
\begin{equation}\label{InnerSum2}
\begin{aligned}
\sum_{\sqrt{x}\leq \ell \leq \frac{x-a}{4u^2}}\left(\frac{P_{a, u}(\ell)}{m}\right)\ell^k 
&= \frac{1}{(k+1)}\left(\frac{x-a}{4u^2}\right)^{k+1}\frac{C_{m}(P_{a, u})}{m}+ O\left(x^{(k+1)/2}+ m\left(\frac{x}{4u^2}\right)^k \right)\\
&= \frac{1}{(k+1)}\left(\frac{x}{4u^2}\right)^{k+1}\frac{C_{m}(P_{a, u})}{m}+ O\left(x^{(k+1)/2}+ m\left(\frac{x}{4u^2}\right)^k \right),
\end{aligned}
\end{equation}
since $|C_{m}(P_{a, u})|\leq m$ and
$$\left(\frac{x-a}{4u^2}\right)^{k+1}= \left(\frac{x}{4u^2}+O(1)\right)^{k+1}=\left(\frac{x}{4u^2}\right)^{k+1} \left(1+O\left(\frac{(k+1)u^2}{x}\right)\right). $$
Thus, inserting the estimate \eqref{InnerSum2} in \eqref{InnerSum1} we deduce that
$$ 
S_u(x)= \frac{x^{k+1}}{4(k+1)u^{k+2}}\sum_{\substack{2< a\leq 4u^2+2\\ d(a, u)\in \D}}\frac{C_{m}(P_{a, u})}{m}+ O\left(m\cdot u^{2-k} x^{k+1/2}\right).
$$ 
Write $m=2^{e_1} p_2^{e_2}\cdots p_n^{e_n}$, where $e_1\geq 0$ and $e_j\geq 1$ for $2\leq j\leq n$, and let $m_0$ be the squarefree part of $m/2^{e_1}$. Then, it follows from Lemma \ref{CharacterSum} that $C_{m}(P_{a, u})\neq 0$ only when $(u, m_0)=1$, and in this case we have
$$ 
\sum_{\substack{2< a\leq 4u^2+2\\ d(a, u)\in \D}}\frac{C_{m}(P_{a, u})}{m}=\frac{(-1)^{\omega(m_0)}}{m_0} \prod_{\substack{2\leq j\leq n\\ e_j \text{ is even }}} \left(1-\frac{2}{p_j}\right) \prod_{\substack{ p_j \mid u\\ e_j \text{ is even }}} \left(1+\frac{1}{p_j-2}\right) \sum_{\substack{2< a\leq 4u^2+2\\ d(a, u)\in \D}}b_{a, u}(m).
$$
Let 
$$B_m(u):=\sum_{\substack{2< a\leq 4u^2+2\\ d(a, u)\in \D}}b_{a, u}(m).$$
Write $u=2^{r_1}u_0$, where $r_1\geq 0$ and $u_0$ is odd. We shall distinguish two cases according to Lemma \ref{CharacterSum}. First, if $e_1=0$ then $b_{a, u}(m)=1$ and hence by Lemma \ref{Raulf} we obtain that 
$$ B_m(u)= N_0(u)+N_1(u)+N_2(u)= \begin{cases} 4\cdot 2^{\eta(u)} & \text{ if } r_1=0 \text{ or } 1,\\
8\cdot 2^{\eta(u)} & \text{ if } r_1=2,\\
16\cdot 2^{\eta(u)} & \text{ if } r_1\geq 3.\\
\end{cases}$$
On the other hand, if $e_1\geq 1$ then we have
$$B_m(u)= N_1(u)+(-1)^{e_1}N_2(u)=\begin{cases} 2(-1)^{e_1}2^{\eta(u)} & \text{ if } r_1=0,\\
4\big(1+(-1)^{e_1}\big)2^{\eta(u)}  & \text{ if } r_1\geq 3,\\
0& \text{ otherwise.}\\
\end{cases}$$
Furthermore, we define
$$F_m(u):=\frac{B_m(u)}{a(m)}\prod_{\substack{ p_j \mid u\\ e_j \text{ is even }}} \left(1+\frac{1}{p_j-2}\right) \text{ where } a(m):=\begin{cases} 4 & \text{ if } m \text{ is odd,}\\
2(-1)^{e_1}  & \text{ if } m \text{ is even.}\end{cases}
$$
Then, one can observe that $B_m(u)/a(m)=2^{\eta(u)}$ if $u$ is odd, and hence it follows that for any fixed $m$, the function $F_m(u)$ is multiplicative in $u$, and $F_m(u)\ll m\cdot d(u)$, where $d(u)=\sum_{\ell \mid u}1$ is the divisor function. 

Combining these estimates with \eqref{DiophantineChange1} and \eqref{DiophantineChange2} we deduce that
\begin{align*} 
\sum_{\substack{d\in \D\\ \ep \leq x}} \chi_d(m) d^{k/2}&= \frac{(-1)^{\omega(m_0)}a(m)}{4m_0}\prod_{\substack{2\leq j\leq n\\ e_j \text{ is even }}} \left(1-\frac{2}{p_j}\right) \frac{x^{k+1}}{(k+1)}\sum_{\substack{u\leq x^{1/10}\\ (u, m_0)=1}} \frac{F_m(u)}{u^{k+2}}\\
&+O\left(m\cdot x^{k+49/50}\right).\\
\end{align*}
Moreover, we have
$$\sum_{\substack{u\leq x^{1/10}\\ (u, m_0)=1}} \frac{F_m(u)}{u^{k+2}}= \sum_{\substack{u\geq 1\\ (u, m_0)=1}} \frac{F_m(u)}{u^{k+2}} +O\left(m\sum_{u>x^{1/10}}\frac{d(u)}{u^2}\right)= \sum_{\substack{u\geq 1\\ (u, m_0)=1}} \frac{F_m(u)}{u^{k+2}}+O\left(m\frac{\log x}{x^{1/10}}\right).
$$ Now, using that $F_m(u)$ is multiplicative, we obtain
\begin{align*} 
\sum_{\substack{u\geq 1\\ (u, m_0)=1}} \frac{F_m(u)}{u^{k+2}}
&=\left(\sum_{a=0}^{\infty} \frac{F_m(2^a)}{2^{a(k+2)}}\right)\prod_{\substack{2\leq j\leq n\\ e_j \text{ is even }}}\left(\sum_{a=0}^{\infty} \frac{F_m(p_j^a)}{p_j^{a(k+2)}}\right)\prod_{p\nmid 2m} \left(\sum_{a=0}^{\infty} \frac{F_m(p^a)}{p^{a(k+2)}}\right)\\
&=\left(\sum_{a=0}^{\infty} \frac{F_m(2^a)}{2^{a(k+2)}}\right)\prod_{\substack{2\leq j\leq n\\ e_j \text{ is even }}}\left(1+\sum_{a=1}^{\infty} \frac{2\left(1+\frac{1}{p_j-2}\right)}{p_j^{a(k+2)}}\right)\prod_{p\nmid 2m} \left(1+\sum_{a=1}^{\infty} \frac{2}{p^{a(k+2)}}\right)\\
&= \left(\sum_{a=0}^{\infty} \frac{F_m(2^a)}{2^{a(k+2)}}\right)\prod_{\substack{2\leq j\leq n\\ e_j \text{ is even }}}\left(1+ \frac{2(p_j-1)}{(p_j-2)(p_j^{k+2}-1)}\right)\prod_{p\nmid 2m} \left(1+ \frac{2}{p^{k+2}-1}\right).\\
\end{align*}
The first factor $\sum_{a=0}^{\infty} F_m(2^a)/2^{a(k+2)}$ depends on whether $m$ is even or odd. Indeed, in the first case (which corresponds to $e_1\geq 1$) we have
$$\sum_{a=0}^{\infty} \frac{F_m(2^a)}{2^{a(k+2)}}= 1+ 2\left(1+(-1)^{e_1}\right)\sum_{a=3}^{\infty}\frac{1}{2^{a(k+2)}}= 1+\frac{2\left(1+(-1)^{e_1}\right)}{4^{k+2}(2^{k+2}-1)}.$$
On the other hand, if $m$ is odd (that is $e_1=0$) then
$$ \sum_{a=0}^{\infty} \frac{F_m(2^a)}{2^{a(k+2)}}= 1+\frac{1}{2^{k+2}}+\frac{2}{4^{k+2}}+4 \sum_{a=3}\frac{1}{2^{a(k+2)}}= 1+\frac{1}{2^{k+2}}+\frac{2}{4^{k+2}}+ \frac{4}{4^{k+2}(2^{k+2}-1)}.$$
Combining these estimates completes the proof.
\end{proof}


\section{Moments of the class number: Proof of Theorem \ref{Main}}
We shall first define the constant $\mathcal{H}(k)$ that appears in the asymptotic formula of Theorem \ref{Main}. For $k\in \mathbb{R}$, let $d_k(m)$ be the $k$-th divisor function, which is the multiplicative function defined on prime powers by $d_k(p^a)=\Gamma(k+a)/(\Gamma(k)a!)$. Then define
\begin{equation}\label{TheH}
\mathcal{H}(k):=\mathcal{C}(k)\cdot \sum_{m=1}^{\infty} \frac{d_k(m)g_k(m)}{m},
\end{equation}
where $\mathcal{C}(k)$ and $g_k$ are defined in Theorem \ref{AsympChar}. Note that 
$\mathcal{H}(k)= \prod_{p} \mathcal{H}_p(k),$
where 
\begin{equation}\label{LocalTerms1}
\mathcal{H}_p(k)= \left(1 + O\left(\frac{1}{p^{k+2}}\right)\right)\sum_{a=0}^{\infty} \frac{d_k(p^a)g_k(p^a)}{p^a},
\end{equation}
by \eqref{TheC}.  We begin by proving the estimates for $\mathcal{H}_p(k)$ that are stated in Theorem \ref{Main}.
\begin{lem}\label{PrimePart}
Let $k$ be a large positive real number. Then we have
$$\mathcal{H}_2(k)=\frac12+ O\left(\frac{2^k}{3^k}\right),$$
and for $p\geq 3$, 
$$ \mathcal{H}_p(k)=\left(\left(\frac12-\frac{3}{2p}\right)\left(1-\frac{1}{p}\right)^{-k}+ \left(\frac12-\frac{1}{2p}\right)\left(1+\frac{1}{p}\right)^{-k}+\frac{2}{p}\right)\left(1+ O\left(\frac{1}{(p-1)^k}\right)\right).
$$
\end{lem}
\begin{proof} Observe that for any prime $p$, and real numbers $k$ and $t$ such that $|t|<p$  we have
\begin{equation}\label{CompMulti}
\sum_{a=0}^{\infty} \frac{d_k(p^a)t^a}{p^a}=\left(1-\frac{t}{p}\right)^{-k}.
\end{equation}
We first consider the case $p=2$. In this case we have
\begin{align*}
\sum_{a=0}^{\infty} \frac{d_k(2^a)g_k(2^a)}{2^a}
&= 1+\sum_{a=1}^{\infty}\frac{d_k(2^a)}{2^a} \left(\frac{(-1)^a}{2} +O\left(\frac{1}{2^k}\right)\right)\\
&= 1+\left(\frac12+ O\left(\frac{1}{2^k}\right)\right)\left(\sum_{a=0}^{\infty}\frac{d_k(2^a)(-1)^a}{2^a}-1\right)= \frac12+ O\left(\frac{2^k}{3^k}\right).
\end{align*}
Next, if $p>2$ then
\begin{align*}
\sum_{a=0}^{\infty} \frac{g_k(p^a)d_k(p^a)}{p^a}&= 1-\frac{1}{p}\sum_{\substack{a\geq 1\\ a \text{ odd }}}\frac{d_k(p^a)}{p^a}+ \left(1-\frac{2}{p}\right)\sum_{\substack{a\geq 1\\ a \text{ even }}}\frac{d_k(p^a)}{p^a}+ O\left(\frac{1}{p^k}\sum_{a=1}^{\infty}\frac{d_k(p^a)}{p^a}\right)\\
&= 1-\frac{1}{p}\sum_{a=1}^{\infty}\frac{d_k(p^a)}{p^a}\left(\frac{1-(-1)^a}{2}\right)+ \left(1-\frac{2}{p}\right)\sum_{a=1}^{\infty}\frac{d_k(p^a)}{p^a}\left(\frac{1+(-1)^a}{2}\right)\\
& \ \ \ \ \ \ +O\left(\frac{1}{(p-1)^k}\right)\\
&=\left(\frac12-\frac{3}{2p}\right)\left(1-\frac{1}{p}\right)^{-k}+ \left(\frac12-\frac{1}{2p}\right)\left(1+\frac{1}{p}\right)^{-k}+\frac{2}{p}+ O\left(\frac{1}{(p-1)^k}\right),
\end{align*}
by \eqref{CompMulti}. The lemma follows upon combining these estimates with \eqref{LocalTerms1}.

\end{proof}

Let $k$ be a positive real number. By the class number formula \eqref{ClassNumberFormula} we have
$$ \sums h(d)^k= \sums \frac{L(1,\chi_d)^k d^{k/2}}{(\log \ep)^k}.$$
Hence, in order to find an asymptotic formula for this moment it suffices to estimate the sum $\sum_{\ep\leq x} L(1,\chi_d)^k d^{k/2}$. 
Indeed, an easy application of partial summation together with the following result imply Theorem \ref{Main}.
\begin{thm}\label{UnweightedMoments}
Let $x$ be large. There exists a positive constant $B$ such that uniformly for all real numbers $k$ with $0<k \leq \log x/(B\log_2 x\log_3 x)$ we have $$\sums L(1,\chi_d)^k d^{k/2} = \frac{\mathcal{H}(k)}{k+1}\cdot x^{k+1} + O\left(x^{k+1}\exp\left(-\frac{\log x}{150\log_2 x}\right)\right).$$
\end{thm}

For any $k\in \mathbb{R}$, and $\re(s)>1$ we have 
$$L(s, \chi_d)^k =\sum_{n=1}^{\infty} \frac{d_k(n)}{n}\chi_d(n).$$
We now recall some standard bounds for the divisor function $d_k(n)$. We have 
$|d_k(n)|\leq d_{|k|}(n)\leq d_{\ell}(n)$
for any integer $\ell\geq |k|$, and $d_{\ell}(mn)\leq d_{\ell}(m)d_{\ell}(n)$ for any positive integers $\ell,m,n$.  Furthermore
for $\ell\in {\Bbb N}$, and $y>3$ we have that
\begin{equation}\label{BoundDivisork0}
 d_{\ell}(n)e^{-n/y}\leq
e^{\ell/y}\sum_{a_1...a_{\ell}=n}e^{-(a_1+...+a_{\ell})/y},
\end{equation} and so
\begin{equation}\label{BoundDivisork}
 \sum_{n=1}^{\infty}\frac{d_{\ell}(n)}{n}e^{-n/y}\leq \left(e^{1/y}
\sum_{a=1}^{\infty}\frac{e^{-a/y}}{a}\right)^{\ell}\leq (\log
3y)^{\ell}.
\end{equation}
One of the key ingredients in the proof of Theorem \ref{UnweightedMoments} is an approximation of $L(1,\chi_d)^k$ by (essentially) a very short Dirichlet polynomial if $L(s,\chi_d)$ has no zeros in a small region to the left of the line $\re(s)=1$. This was proved in Proposition 3.3 of \cite{DaLa}. 
\begin{pro}[Proposition 3.3 of \cite{DaLa}]\label{ShortApproxL}
Let $d$ be a large positive discriminant and $0<\epsilon<1/2$ be fixed. Let $y$ be a real number such that $\log d/\log_2 d\leq \log y\leq \log d$. Furthermore, assume that $L(s,\chi_d)$ has no zeros inside the rectangle $\{s:1-\epsilon <\textup{Re}(s)\leq 1 \text{ and } |\textup{Im}(s)|\leq 2(\log d)^{2/\epsilon}\}$.
Then for any positive real number $k$ such that $k\leq  \log y/(4\log_2 d \log_3 d)$ we have
$$L(1,\chi_d)^k=\sum_{n=1}^{\infty}\frac{d_k(n)\chi_d(n)}{n}e^{-n/y}+O_{\epsilon}\left(\exp\left(-\frac{\log y}{2\log_2 d}\right)\right).$$
\end{pro}

Our last ingredient is a bound on the number of discriminants $d\in \D$ with $\ep\leq x$ and such that $L(s, \chi_d)$ has a zero in the rectangle
$$ \mathcal{R}(x):=\left\{s\in \mathbb{C}: 9/10<\text{Re}(s)<1\text{ and }|\im(s)|\leq (\log x)^{40}\right\}.$$
\begin{lem}\label{ZeroDensity}
There are at most $O(x^{1-1/20})$ discriminants $d\in \D$ such that $\ep\leq x$ and $L(s,\chi_d)$ has a zero in the rectangle $\mathcal{R}(x)$.
\end{lem}

\begin{proof}
First, by Lemma \ref{Correspondence}  we have
\begin{align*}
&\left|\{(d, n): d\in \mathcal{D}, n\in \mathbb{N}, \ep^n\leq x \text{ and } L(s,\chi_d) \text{ has a zero in } \mathcal{R}(x)\}\right|\\
= &\left|\{(t, u): u\in \mathbb{N},  2<t<X \text{ and } L(s,\chi_{d(t, u)}) \text{ has a zero in } \mathcal{R}(x)\}\right|.
\end{align*} 
The contribution of the pairs $(d, n)$ such that $n\geq 2$, $d\in \D$ and $\ep^n\leq x$ is negligible. Indeed, we have 
$$ \sum_{d\in \D}\sum_{\substack{n\geq 2 \\ \ep^n\leq x}}1=\sum_{2\leq n\ll \log x}\sum_{\substack{d\in \D \\ \ep\leq x^{1/n}}}1\ll x^{\frac12}\log x,$$
by \eqref{SarnakAsymp}. Hence, combining this bound with \eqref{RaulfBound} we deduce that
\begin{align*}
&\left|\{d\in \mathcal{D}:  \ep\leq x \text{ and } L(s,\chi_d) \text{ has a zero in } \mathcal{R}(x)\}\right|\\
= &\left|\{(t, u): u\leq x^{\delta},  2<t<x \text{ and } L(s,\chi_{d(t, u)}) \text{ has a zero in } \mathcal{R}(x)\}\right| +O\left(x^{1-\delta/5}\right).
\end{align*} 
In order to bound the number of pairs $(t, u)$ for which $L(s,\chi_{d(t, u)})$ has a zero in $\mathcal{R}(x)$, we shall use the following zero-density result of Heath-Brown \cite{He}, which states that for $1/2<\sigma<1$ and
 any $\epsilon>0$ we have
\begin{equation}\label{Heath}
\sumf_{|d|\leq x} N(\sigma,T, \chi_d)\ll (xT)^{\epsilon}x^{3(1-\sigma)/(2-\sigma)}T^{(3-2\sigma)/(2-\sigma)},
\end{equation}
where $N(\sigma, T, \chi_d)$ is the number of zeros $\rho$ of $L(s,\chi_d)$ with $\re(\rho)\geq \sigma$ and $|\im(\rho)|\leq T$, and $\sumf$ indicates that the sum is over fundamental discriminants.

Let $T=(\log x)^{40}$, and $M$ be the number of positive fundamental discriminants $d\leq x^2$  such that $L(s,\chi_d)$ has a zero in the rectangle $\mathcal{R}(x)$. Then, it follows from \eqref{Heath} that 
\begin{equation}\label{BoundM}
M \leq \sumf_{|d|\leq x^2} N\left(9/10, T, \chi_d\right)\ll x^{3/5}.
\end{equation}

Recall that for every discriminant $d\in \D$, there exists a unique fundamental discriminant $\widetilde{d}$ such that $d=\widetilde{d}\cdot\ell^2$ for some $\ell\in \mathbb{N}$. Moreover, the zeros of $L(s,\chi_d)$ and those of $L(s, \chi_{\widetilde{d}})$ are the same in the half-plane $\text{Re}(s)>0$.  
Let $d_1, \cdots, d_M$ be the positive fundamental discriminants $d\leq x^2$ for which $L(s,\chi_d)$ has a zero in $\mathcal{R}(x)$, and  fix $u\leq x^{\delta}$. If $2<t\leq x$ and $L(s,\chi_{d(t, u)})$  has a zero in $\mathcal{R}(x)$ then $d(t, u)\leq x^2$ and hence there exists a unique $j\in \{1,2, \dots, M\}$ such that $d(t, u)=d_j \cdot \ell^2$ for some $\ell\in \mathbb{N}$. This implies that $t^2-4= d_j u^2 \ell^2$, which shows that $(t, \ell)$ is a solution to Pell's equation $x^2-d_ju^2y^2=4$, and therefore we must have 
$$\frac{t+\ell\sqrt{d_j}u}{2}=  \varepsilon_{d_j u^2}^n, \text{ for some } n\in \mathbb{N},$$
where $\varepsilon_{d_j u^2}$ is the fundamental unit associated to the discriminant $d_j u^2$. Since $\varepsilon_{d_j u^2}^n\leq t$ we obtain that
$$|\{ (t, \ell):  2<t\leq x, \ell\geq 1, \text{ and } t^2-4= d_j u^2 \ell^2\}|\leq  |\{n \in \mathbb{N}: \varepsilon_{d_j u^2}^n\leq x\}\ll \log x.
$$
Thus, we deduce that
$$ \left|\{d\in \mathcal{D} : \ep\leq x \text{ and } L(s,\chi_d) \text{ has a zero in } \mathcal{R}(x)\}\right|\ll x^{\delta} M\log x+ x^{1-\frac{\delta}{5}}.$$
Choosing $\delta=1/4$ and using the bound \eqref{BoundM} completes the proof. 

\end{proof}

We are now ready to prove Theorem \ref{UnweightedMoments}.

\begin{proof} [Proof of Theorem \ref{UnweightedMoments}]
Let $d\in \D$ such that $\ep\leq x$. Then we have $d\leq \ep^2\leq x^2$, and thus for any real number $k$ such that $0<k\ll \log x/(\log_2 x \log_3 x)$, we have
\begin{equation}\label{StandardBoundL}
L(1, \chi_d)^k\leq (c\log d)^{k}\leq  (2c\log x)^k=x^{o(1)},
\end{equation}
for some constant $c>0$, which follows from the standard bound $L(1,\chi_d)\ll \log d$. 

Let $\widetilde{\D}(x)$ be the set of discriminants $d\in \D$ such that $d\geq \sqrt{x}$, $\ep\leq x$ and $L(s,\chi_d)$ does not have a zero in the rectangle $\mathcal{R}(x)$. Then, it follows from Lemma \ref{ZeroDensity} that 
\begin{equation}\label{BoundExceptional}
|\{ d\in\D: \ep\leq x\}|- |\widetilde{\D}(x)|\ll x^{1-1/20}.
\end{equation}
Therefore, using this bound together with \eqref{StandardBoundL} we obtain
$$\sums L(1,\chi_d)^k d^{k/2}= \sum_{d\in \widetilde{\D}(x)}L(1,\chi_d)^k d^{k/2}+ O\left(x^{k+1-1/30}\right).$$
Let $y=x^{1/60}$, and $\ell=[k]+1$. Then, it follows from Proposition \ref{ShortApproxL} that for any $d\in \widetilde{\D}(x)$ we have
$$ 
 L(1,\chi_d)^k= \sum_{n=1}^{\infty}\frac{d_k(n)\chi_d(n)}{n}e^{-n/y}+O\left(\exp\left(-\frac{\log x}{150\log_2 x}\right)\right),
$$
if $k\leq \log x/(B\log_2\log_3 x)$ with a suitably large constant $B$. Hence, we derive
$$
\sums L(1,\chi_d)^k d^{k/2}= \sum_{n=1}^{\infty}\frac{d_k(n)}{n}e^{-n/y}\sum_{d\in \widetilde{\D}(x)}\chi_d(n)d^{k/2}+ O\left(x^{k+1}\exp\left(-\frac{\log x}{150\log_2 x}\right)\right).
$$
We now extend the main term of the last estimate, so as to include all discriminants $d\in \D$ with $\ep\leq x$. Using  \eqref{BoundDivisork} and   \eqref{BoundExceptional}, we deduce that 
$$ \sum_{n=1}^{\infty}\frac{d_k(n)}{n}e^{-n/y}\left(\sums\chi_d(n)d^{k/2}-\sum_{d\in \widetilde{\D}(x)}\chi_d(n)d^{k/2}\right)\ll x^{k+1-1/20}\sum_{n=1}^{\infty}\frac{d_{\ell}(n)}{n}e^{-n/y}\ll x^{k+1-1/30}.
$$
Combining this bound with Theorem \ref{AsympChar} we obtain
\begin{equation}\label{EstimateIntermediate}
\begin{aligned}
\sums L(1,\chi_d)^k d^{k/2}&= \frac{\mathcal{C}(k)}{k+1} \cdot x^{k+1}\cdot \sum_{n=1}^{\infty}\frac{d_k(n)g_k(n)}{n}e^{-n/y}\\
&+O\left(x^{k+1-1/50}\sum_{n=1}^{\infty}d_{\ell}(n)e^{-n/y}+ x^{k+1}\exp\left(-\frac{\log x}{150\log_2 x}\right)\right).
\end{aligned}
\end{equation}
To bound the error term in the last estimate, we split the sum $\sum_{n=1}^{\infty}d_{\ell}(n)e^{-n/y}$ into two parts: $n\leq y \log^2 y$ and $n> y \log^2 y$. The contribution of the first part is 
$$
 \leq  \sum_{n\leq y\log^2 y} \frac{y \log^2 y}{n} d_{\ell}(n) e^{-n/y}
 \leq y\log^2 y\sum_{n=1}^{\infty} \frac{d_{\ell}(n)}{n} e^{-n/y}
\ll y(\log 3y)^{\ell+2}\ll x^{1/59},
$$
by \eqref{BoundDivisork}, if $x$ is sufficiently large.
The remaining terms contribute
\begin{align*}
 \leq  \exp\left(-\frac{(\log y)^2}{2}\right)\sum_{n=1}^{\infty}d_{\ell}(n)e^{-n/(2y)}
 & \leq  \exp\left(-\frac{(\log y)^2}{2}\right)\left(e^{1/y}\sum_{a=1}^{\infty}e^{-a/y}\right)^{\ell}\\
 & \ll  \exp\left(-\frac{(\log y)^2}{2}\right) (2y)^{\ell} \ll \exp\left(-\frac{(\log y)^2}{4}\right),
\end{align*}
by \eqref{BoundDivisork0}. Combining these estimates and using that $\mathcal{C}(k)=O(1)$ we deduce that  
\begin{equation}\label{EstimateIntermediate}
\sums L(1,\chi_d)^k d^{k/2}= \frac{\mathcal{H}(k)}{k+1} \cdot x^{k+1}+ O\left(E_k(x)+ x^{k+1}\exp\left(-\frac{\log x}{150\log_2 x}\right)\right),
\end{equation}
where 
$$
E_k(x)= x^{k+1}\cdot \sum_{n=1}^{\infty}\frac{d_k(n)|g_k(n)|}{n}\left(1-e^{-n/y}\right).
$$
A simple computation together with the definition of $g_k(p)$ shows that for every real number $k>0$ we have
$\left|g_k(2^a)\right|\leq 1/2$ for all $a\geq 1$, and for $p>2$ and $a\geq 1$ we have
$$\left|g_k(p^a)\right|\leq \begin{cases} \frac{1}{p} &\text{ if } a \text{ is odd},\\
1 &\text{ if } a \text{ is even}.\\
\end{cases}$$
Thus, it follows that for every real number $k>0$ and positive integer $n$ we have  $|g_k(n)|\leq 1/n_0$ where $n_0$ is the squarefree part of $n$. Next, note that $1-e^{-n/y}\ll (n/y)^{\alpha}$ for all real numbers $0<\alpha\leq 1$.  Choosing $\alpha= 1/\log\log x$, and writing $n=n_0 n_1^2$ we deduce that  
\begin{align*}
E_k(x) &\ll y^{-\alpha}x^{k+1} \sum_{n=1}^{\infty} \frac{d_{\ell}(n)}{n_0 n^{1-\alpha}}\\
&\ll y^{-\alpha}x^{k+1}\sum_{n_0=1}^{\infty}
\frac{d_{\ell}(n_{0})}{n_{0}^{2-\alpha}}\sum_{n_1=1}^{\infty}\frac{d_{\ell}^{2}(n_{1})}{n_{1}^{2-2\alpha}}= y^{-\alpha} x^{k+1}\zeta(2-\alpha)^{\ell} \sum_{n=1}^{\infty}\frac{d_{\ell}^{2}(n)}{n^{2-2\alpha}},
\end{align*}
since $d_{\ell}(n_0n_1^2)\leq d_{\ell}(n_0)d_{\ell}(n_1^2)\leq  d_{\ell}(n_0)d_{\ell}(n_1)^2.$
Finally, we use the following bound, which follows from Lemma 3.3 of \cite{La1} 
\[
 \sum_{n=1}^{\infty}\frac{d_{\ell}^{2}(n)}{n^{2-2\alpha}} \leq \exp\Big(\big(2+o(1)\big) \ell\log_2 \ell \Big).
 \]
This shows that  
$ E_k(x) \ll x^{k+1}\exp\left(-\log x/(100\log_2 x)\right),$
which completes the proof.

\end{proof}
\section{The distribution of large values of $h(d)$: Proof of Theorem \ref{Distribution}}
In order to use Theorem \ref{Main} to study the distribution of large values of $h(d)$, we first need to estimate the constant $\mathcal{H}(k)$ when $k$ is large. We prove
\begin{pro}\label{LeadingConstant}
Let $k$ be a large positive real number. Then we have
$$\log \mathcal{H}(k)= k\log\log k+ k(\gamma-\log 3)+ (A_0-1) \frac{k}{\log k}+ O\left(\frac{k}{(\log k)^2}\right),$$
where $A_0$ is defined in Theorem \ref{Distribution}.
\end{pro}
Recall that $\mathcal{H}(k)=\prod_{p} \mathcal{H}_p(k)$. Furthermore, it follows from  Lemma \ref{PrimePart} that $\mathcal{H}_2(k)=O(1)$ and $\mathcal{H}_3(k)=O(1)$ for all $k>1$. Hence, we only need to estimate $\mathcal{H}_p(k)$ for primes $p\geq 5$.

\begin{lem}\label{PrimeSumEstimate}
Let $k$ be a large positive real number, and $p$ be a prime number. Then we have
$$\log \mathcal{H}_p(k)= 
\begin{cases} -k\cdot \log\left(1-\frac{1}{p}\right)+O(1) & \text{ if } 5\leq p\leq k^{2/3}, \\ \log\cosh\left(\frac{k}{p}\right)+O\left(\frac{k}{p^2}\right) & \text{ if } p> k^{2/3}.
\end{cases}
$$
\end{lem}

\begin{proof} First, if $5\leq p\leq k^{2/3}$ then it follows from Lemma \ref{PrimePart} that
$$ \mathcal{H}_p(k)=\left(\frac{1}{2} -\frac{3}{2p}\right)\left(1-\frac 1p\right)^{-k} \left(1+O\left(\exp\left(-k^{1/3}\right)\right)\right).
$$
Taking the logarithm of both sides implies the desired estimate in this case. Now, if 
$p>k^{2/3}$ then by Lemma \ref{PrimePart} we have
\begin{align*}
\mathcal{H}_p(k)&= \left(\left(\frac12-\frac{3}{2p}\right) e^{k/p}+ \left(\frac12-\frac{1}{2p}\right)e^{-k/p}+\frac{2}{p}\right)\left(1+ O\left(\frac{k}{p^2}\right)\right)\\
&= \left(\left(1-\frac{3}{p}\right)\cosh\left(\frac{k}{p}\right) +\frac{1}{p} e^{-k/p}+\frac{2}{p}\right)\left(1+O\left(\frac{k}{p^2}\right)\right)\\
&= \cosh\left(\frac{k}{p}\right) \left(1+O\left(\frac{k}{p^2}\right)\right),
\end{align*}
since $\cosh(t)-1\leq \sinh(t)\ll t\cosh(t)$ for $t>0$. This completes the proof.
\end{proof}

Define 
$$ f(t):= \begin{cases} \log \cosh(t) & \text{ if } 0\leq t <1, \\
 \log \cosh(t)- t  & \text{ if } t \geq 1.\end{cases}
$$
Then, we have 
\begin{equation}\label{logcosh} 
f(t):= \begin{cases} t^2/2+O(t^4) &  \text{ if } 0\leq t <1, \\
 O(1) & \text{ if } t \geq 1.\end{cases}
\end{equation}
We are now ready to prove Proposition \ref{LeadingConstant}. 
\begin{proof}[Proof of Proposition \ref{LeadingConstant}]
First, recall that $\mathcal{H}_2(k)=O(1)$ and $\mathcal{H}_3(k)=O(1)$ for all $k>1$, by Lemma \ref{PrimePart}.  Therefore, it follows from  Lemma \ref{PrimeSumEstimate} and equation \eqref{logcosh} that
\begin{equation}\label{SumPrimesEstimate}
\begin{aligned}
\log \mathcal{H}(k) &=-k\sum_{5\leq p\leq k^{2/3}}\log\left(1-\frac{1}{p}\right)+\sum_{p>k^{2/3}}\log\cosh\left(\frac{k}{p}\right)+ O\left(k^{2/3}\right)\\
&=-k\sum_{5\leq p\leq k}\log\left(1-\frac{1}{p}\right)+\sum_{k^{2/3}<p<k^{4/3}} f\left(\frac{k}{p}\right)+ O\left(k^{2/3}\right),
\end{aligned}
\end{equation}
since $f(t)=t^2/2+O(t^4)$ for $0\leq t <1.$
Now, using the prime number theorem in the form $\pi(t)-\text{Li}(t)\ll t/(\log t)^3$, together with partial summation, we obtain
\begin{equation}\label{PartialSummation1}
\begin{aligned}
\sum_{k^{2/3}<p<k^{4/3}}f\left(\frac{k}{p}\right) &=\int_{k^{2/3}}^{k^{4/3}} f\left(\frac{k}{t}\right)\frac{dt}{\log t}+O\left(\frac{k}{(\log k)^2}\right)\\
&=\frac{k}{\log k}\int_{k^{-1/3}}^{k^{1/3}}\frac{f(u)}{u^2}du+O\left(\frac{k}{(\log k)^2}\right),
\end{aligned}
\end{equation}
by a change of variables $u=k/t$, and since $\int_0^{\infty}(f(u)(\log u)/u^2)du<\infty$. Extending the integral in the right hand side of this estimate gives
\begin{equation}\label{PartialSummation2}
 \int_{k^{-1/3}}^{k^{1/3}}\frac{f(u)}{u^2}du= \int_0^{\infty} \frac{f(u)}{u^2}du +O\left(k^{-1/3}\right).
\end{equation}
by \eqref{logcosh}. The result follows upon combining \eqref{SumPrimesEstimate}, \eqref{PartialSummation1} and \eqref{PartialSummation2}, together with the following estimate which follows from the prime number theorem
$$
-\sum_{5\leq p\leq k} \log \left(1-\frac1p\right)= \log\log k+\gamma-\log 3 +O\left(\frac{1}{(\log k)^2}\right).
$$ 

\end{proof}
We end this section by proving Theorem \ref{Distribution}. To shorten our notation, we let $\D(x)$ be the set of $d\in \D$ with $\ep\leq x$. We also let $\mathcal{N}_x(\tau)$ be the proportion of discriminants $d\in \D(x)$ such that 
$$h(d)\geq \frac{e^{\gamma}}{3} \frac{x}{\log x}\cdot \tau.$$

\begin{proof}[Proof of Theorem \ref{Distribution}]
First, observe that for every real number $k>1$ we have
\begin{align*}
k\int_0^{\infty} t^{k-1} \N_x(t)dt&= \frac{1}{|\D(x)|}\sum_{d\in \D(x)}\int_0^{3e^{-\gamma}h(d) (\log x)/x} k t^{k-1} dt\\
& = \left(3 e^{-\gamma} \frac{\log x}{x}\right)^k \frac{1}{|\D(x)|}\sum_{d\in \D(x)} h(d)^k.
\end{align*} 
Therefore, it follows from \eqref{SarnakAsymp} together with Corollary \ref{Main2} and Proposition \ref{LeadingConstant} that for $1\ll k\leq \log x/(B\log_2 x\log_3 x)$, we have
\begin{equation}\label{EstimatePowerN}
\int_0^{\infty} t^{k-1} \N_x(t)dt= (\log k)^k \exp\left(\frac{k}{\log k} \left(A_0-1+O\left(\frac{1}{\log k}\right)\right)\right).
\end{equation}
Define $\kappa:= e^{\tau-A_0}$ so that $\tau=\log \kappa+A_0$. Also, let $K= \kappa e^{\delta}$ where $\delta>0$ is a small parameter to be chosen later. Then, it follows from \eqref{EstimatePowerN} that 
\begin{align*}
\int_{\tau+\delta}^{\infty} t^{\kappa-1} \N_x(t)dt &\leq (\tau +\delta)^{\kappa (1-e^{\delta})}\int_0^{\infty} t^{K-1} \N_x(t)dt\\
&= (\tau +\delta)^{\kappa (1-e^{\delta})} (\log\kappa +\delta)^{\kappa e^{\delta}} \exp\left(\frac{\kappa e^{\delta}}{\log \kappa}\left(A_0-1+O\left(\frac{1}{\log \kappa}\right)\right)\right)\\
&= (\log \kappa)^{\kappa} \exp\left(\frac{\kappa}{\log \kappa} \left(A_0-1 +(\delta+1-e^{\delta}) +O\left(\frac{1}{\log \kappa}\right)\right)\right),
\end{align*}
by the definition of $\kappa$. Choosing $\delta= C_0/\sqrt{\log k}$ for a suitably large constant $C_0$ implies
$$
\int_{\tau+\delta}^{\infty} t^{\kappa-1} \N_x(t)dt \leq  \left(\int_{0}^{\infty} t^{\kappa-1} \N_x(t)dt\right) \exp\left(-\frac{\kappa}{\log \kappa^2}\right). 
$$
A similar argument shows that
$$
\int_{0}^{\tau-\delta} t^{\kappa-1} \N_x(t)dt \leq  \left(\int_{0}^{\infty} t^{\kappa-1} \N_x(t)dt\right) \exp\left(-\frac{\kappa}{\log \kappa^2}\right). 
$$
Combining these bounds with \eqref{EstimatePowerN} gives
\begin{equation}\label{CuttingTails}
\int_{\tau-\delta}^{\tau+\delta} t^{\kappa-1} \N_x(t)dt= (\log \kappa)^{\kappa} \exp\left(\frac{\kappa}{\log \kappa} \left(A_0-1+O\left(\frac{1}{\log \kappa}\right)\right)\right).
\end{equation}
Furthermore, since $\N_x(t)$ is non-increasing as a function of $t$ we can bound the above integral as follows
$$ 
\tau^{\kappa} \exp\left(O\left(\frac{\delta\kappa}{\tau}\right)\right) \N_x(\tau+\delta)
\leq \int_{\tau-\delta}^{\tau+\delta} t^{\kappa-1} \N_x(t)dt\leq  \tau^{\kappa} \exp\left(O\left(\frac{\delta\kappa}{\tau}\right)\right) \N_x(\tau-\delta).
$$ 
Inserting these bounds in \eqref{CuttingTails}, and using the definition of $\kappa$ we obtain
$$  \N_x(\tau+\delta) \leq  \exp\left(-\frac{e^{\tau-A_0}}{\tau} \left(1+O\left(\delta\right)\right)\right) \leq  \N_x(\tau-\delta),
$$
and thus 
$$ \N_x(\tau)= \exp\left(-\frac{e^{\tau-A_0}}{\tau} \left(1+O\left(\frac{1}{\sqrt{\tau}}\right)\right)\right).$$
as desired.
\end{proof}

\section*{Acknowledgements} I would like to thank the anonymous referee for many valuable suggestions which helped strengthen the results of the paper. I would also like to thank Gergely Harcos for a useful comment.

\end{document}